\documentclass{article}
\usepackage{amsmath, amssymb, xcolor, amsthm}
\usepackage{bbm}
\usepackage{graphicx,pdfsync}
\usepackage{amsxtra,amssymb,amsmath,enumerate,graphicx,bbm,hyperref,amsthm,color}
\usepackage[active]{srcltx}
\usepackage{t1enc}
 
\usepackage[T1]{fontenc}

\allowdisplaybreaks

\newtheorem{proposition}{Proposition}[section]

\newtheorem{remark}[proposition]{Remark}

\newtheorem{theorem}[proposition]{Theorem}
\newtheorem{assumption}[proposition]{Assumption}
\newtheorem{example}[proposition]{Example}

\newcommand{\E}{\mathbb{E}}

\renewcommand{\P}{\mathbb{P}}

\newcommand{\R}{\mathbb{R}}

\newcommand{\1}{\mathbf{1}}

\usepackage{color}
\definecolor{gray97}{gray}{.97}
\definecolor{gray75}{gray}{.75}
\definecolor{gray45}{gray}{.45}
\definecolor{gray35}{gray}{.35}
\def\x{\times}

\def\Lc{{\cal L}}
\def\vp{\varphi}
\def\E{{\mathbb E}}
\def\eps{\varepsilon}
\def\Lc{{\cal L}}

\title{Monte-Carlo methods for the pricing of American options: a semilinear BSDE point of view}

\author{Bruno Bouchard\footnote{{Universit\'e Paris-Dauphine, PSL Research University, CNRS, CEREMADE, Paris.}  {bouchard@ceremade.dauphine.fr}. Research of B.~Bouchard partially supported by ANR CAESARS (ANR-15-CE05-
0024). } \and Ki Wai Chau\footnote{Centrum Wiskunde \& Informatica. K.W.Chau@cwi.nl} \and  Arij Manai\footnote{Institut du Risque et de l'Assurance du Mans, Le Mans universit\'{e}. arijmanai@gmail.com} \and  Ahmed Sid-Ali\footnote{Universit\'{e} Laval, D\'{e}partement de math\'{e}matiques et de statistique, Qu\'{e}bec, Canada. ahmed.sid-ali.1@ulaval.ca}
}

\begin{document}
\maketitle

\begin{abstract} 
We extend the viscosity solution characterization proved in \cite{benth2003semilinear} for call/put American option prices to  the case of a  general payoff function in a multi-dimensional setting: the price satisfies a semilinear reaction/diffusion type equation. Based on this, we propose  two  new numerical schemes inspired by  the branching processes based algorithm of \cite{bouchard2016numerical}. Our numerical experiments show that approximating the discontinuous driver of the associated reaction/diffusion PDE  by local polynomials is not efficient,  while a simple {randomization} procedure provides very good results.
\end{abstract}

\vspace{5mm}

\noindent {\bf Keywords~:} American options,  Viscosity solution, Semilinear Black and Scholes partial differential equation, Branching method, BSDE.

\vspace{5mm}

\pagestyle{myheadings}
\thispagestyle{plain}
 
 \section{Introduction} 

An American option is a financial contract which can be exercised by its holder at any time until a given future date, called maturity. 
When it is exercised, the holder receives a payoff that depends on the value of the underlying assets. 
 
Putting this problem in a mathematical context, let us first consider the case of a single stock (non-dividend paying) market under the famous Black and Scholes setting,   \cite{black1973pricing}.
Namely, let $(\Omega, \mathcal{F}, ({\cal F}_{t})_{t\ge 0},\mathbb{P})$ be a filtered probability space carrying a standard one dimensional Brownian motion $W$ and let us model the stock price process $X$ as 
\begin{equation*}
X_{s}=x\exp\big((r-\frac{\sigma^{2}}{2})(s-t)+\sigma (W_{s}-W_{t})\big),\;s\ge t,
\end{equation*}
under the risk natural probability. Here, $x>0$ is the stock price at time $t$,  $r>0$ is the risk-free interest rate and $\sigma >0$ is the volatility. 
Then, the arbitrage free value  at time $t$ of an American option maturing at $T\ge t$ is given by 
\begin{align}
\label{AmOp}
V(t,x)=\sup_{\tau\in {\cal T}_{[t,T]}}\mathbb{E}[e^{-r(\tau-t)}g(X_{\tau})]
\end{align}
where ${\cal T}_{[t,T]}$ is the collection of $[t,T]$-valued stopping times,  and  $g$ is the payoff function, say continuous, see e.g.~\cite{bouchard2016fundamentals} and the references therein.
Typical examples are  
$$
g(x')=
\begin{cases} (x'-K)^+,\,\,\,\,\, \mbox{for a call option} \\
(K-x')^+,\,\,\,\,\,\,\, \mbox{for a put option,}
\end{cases}
$$ 
where $K>0$ denotes the strike price.
 
 By construction, $V(\cdot, X)\ge g(X)$, and the option should be exercised only when $V(\cdot, X) {=} g(X)$. This leads to define the following two regions:
\begin{itemize}
\item the continuation region: 
\begin{equation*}
\textit{C}=\{(t,x) \in [0,T) \times (0,\infty): V(t,x)>g(x)\}
\end{equation*}
\item the stopping (or the exercise) region:
\begin{equation*}
\textit{S}=\{(t,x) \in [0,T) \times (0,\infty):\,V(t,x)=g(x)\}. 
\end{equation*} 
\end{itemize}

These are the basics of the  common formulation of the American option price as a free boundary problem, which already appears in McKean   \cite{mckean1965free}: $V$ solves a heat-equation type linear parabolic problem on $C$ and equals $g$ on $S$, with the constraint of being always greater than $g$. 
Another formulation is based on the quasi-variational approach of Bensoussan and Lions   \cite{bensoussan2011applications}:  the  price solves (at least in the viscosity solution sense) the quasi-variational partial differential equation
\begin{equation*}
\begin{cases} \min\left(r \varphi-\mathcal{L}_{BS}\varphi, \varphi-g\right)=0,\,\,\,\,\, \mbox{ on } [0,T)\times (0,\infty)\\
\varphi(T,\cdot)=g,\hspace{2.65cm}\,\,\, \mbox{ on } (0,\infty)
\end{cases}
\end{equation*}
in which  $\mathcal{L}_{BS}$ is the Dynkin operator associated to $X$:
\begin{equation*}
\mathcal{L}_{BS}=\partial_t+rxD+\frac{1}{2}\sigma^2x^2D^{2}
\end{equation*}
where $D$ and $D^{2}$ are the Jacobian and Hessian operators. 
 
In this paper, we  focus on another formulation that can be found in \cite{benth2003semilinear}, see also \cite{benth2004semilinear} and the references therein. The American   option valuation problem can be stated in terms of a semilinear Black and Scholes partial differential equation set on a fixed domain, namely:
\begin{equation}\label{eq: react diff}
\begin{cases} r\varphi- \mathcal{L}_{BS}\varphi=q(\cdot,\varphi),\,\, \mbox{ on } [0,T) \times (0,+\infty)  \\
\varphi(T,\cdot)=g,\hspace{1.5cm} \mbox{ on } (0, \infty)
\end{cases}
\end{equation}
where  $q$ is a nonlinear reaction term defined as 
\begin{align*}
q(x,\varphi(t,x))=c(x)H(g(x)-\varphi(t,x))=\left\{ \begin{array}{rcl}
0 & \mbox{if} & g(x)< \varphi(t,x) \\ 
c(x) & \mbox{if} & g(x)\geq \varphi(t,x), \\
\end{array}\right.
\end{align*}
in which  $c$ is a certain cash flow function, e.g.~$c=rK$ for a put option, and $H$ is the Heaviside function.

Note that this semilinear Black and Scholes equation does not make sense if we consider classical solutions because of  the discontinuity of $y \rightarrow q(x,y)$. 
It has to be considered in the discontinuous viscosity solution sense, see e.g.~Crandall, Ishii and Lions \cite{crandall1992user}. Namely, even if $V$ is continuous, the supersolution property should be stated in terms of the {lower}-semicontinuous envelope of $q$, the other way round for the subsolution property. This means in particular that the super- and subsolution properties are not defined with respect to the same operator. Still, thanks to the very specific monotonicity of $y \rightarrow q(x,y)$, it is proved in \cite{benth2003semilinear} that, within the Black and Scholes model,  the American option price in the unique solution of \eqref{eq: react diff} in the appropriate sense.

In this work, we first extend  the characterization of  \cite{benth2003semilinear} in terms of  \eqref{eq: react diff} to a general payoff function and to a general market model,  see Section 2. 
Then, we suggest two numerical schemes based on this formulation. The general idea consists in (formally)
 identifying the solution $V$ of \eqref{eq: react diff} to the solution $(Y,Z)$ of the backward stochastic differential equation
 $$
 Y=e^{-rT}g(X_{T})+\int_{\cdot}^{T} e^{-rs}q(X_{s},e^{rs}Y_{s})ds-\int_{\cdot}^{T} Z_{s}dW_{s}
 $$
 by $e^{-r\cdot}V(\cdot,X)=Y$. 
In the  first algorithm, we follow the approach of Bouchard et al.~\cite{bouchard2016numerical} and approximate the nonlinear driver $q$ by local polynomials so as to be able to apply an extended version of the pure forward  branching processes based Feynman-Kac representation of the Kolmogorov-Petrovskii-Piskunov equation, see  \cite{henry2012cutting,henry2016branching}. Unfortunately, our numerical experiments show that this algorithm is quite unstable, see Section \ref{sec: loc poly approach}. In the second algorithm, we do not try to approximate $q$ by local polynomials but in place regularize it with a noise by replacing $q(X,e^{r\cdot}Y)$ by $c(X){\bf 1}_{\{g(X)+\epsilon\ge e^{r\cdot}Y\}}$, in which $\epsilon$ is an independent random variable. When the variance of $\epsilon$ vanishes, this provides a converging estimator. For $\epsilon$ given, 
the corresponding $Y$ is estimated by using the approach of Bouchard et al.~\cite{bouchard2016numerical} with (random) polynomial $(t,x,y,y')\mapsto c(x){\bf 1}_{\{g(x)+\epsilon\ge  e^{rt}y'\}}$ and particles that can only  die (without  creating any children).  This algorithm turns out to be very precise, see Section \ref{sec: indic approach}.

\section{Non-linear parabolic equation representation}

From now on, we take $\Omega$ as the space of $\R^{d}$-valued continuous maps on $[0,T]$ starting at $0$, endowed with the Wiener measure $\P$. We let $W$ denote the canonical process and let $({\cal F}_{t})_{t\le T}$ be its completed filtration.  Given $t\in [0,T]$ and $x\in (0,\infty)^{d}$, we consider a financial market with $d$ stocks whose prices process $X^{t,x}$ evolves according to 
\begin{align}\label{eq: sde}
X^{t,x}=x+\int_{t}^{\cdot} rX^{t,x}_{s} ds +\int_{t}^{\cdot}\sigma(s,X^{t,x}_{s}) dW_{s}
\end{align}
in which $r\in \R$ is a constant\footnote{It should be clear that this assumption is  only made for simplicity. Also note that a dividend rate could be added at no cost.}, the risk free interest rate, and $\sigma:  [0,T]\times (0,\infty)^{d}\mapsto \R^{d\times d}$ is a matrix valued-function that is assumed to be continuous and uniformly Lipschitz in its  second component. We also assume that $\bar \sigma: (t',x')\in [0,T]\times (0,\infty)^{d}\mapsto {\rm diag}[x']^{-1}\sigma(t',x')$ is uniformly Lipschitz in its second component and bounded, where $ {\rm diag}[x']$ stands for the diagonal matrix with $i$-th diagonal entry equal to the $i$-th component of $x'$. 
This implies that $X^{t,x}$ takes values in $(0,\infty)^{d}$ whenever $x\in (0,\infty)^{d}$.

  We also assume that $\P$ is the only (equivalent) probability measure under which $e^{-r(\cdot-t)}X^{t,x}$ is a (local) martingale, for $(t,x)\in [0,T]\times(0,\infty)^{d}$. Then, given a continuous payoff function $g: (0,\infty)^{d}\to \R$, with polynomial growth, the price of the American option with payoff $g$ is given by 
 \begin{eqnarray*}
V(t,x)=\sup_{\tau\in {\cal T}_{[t,T]}}\mathbb{E}[e^{-r(\tau-t)}g(X^{t,x}_{\tau})], 
 \end{eqnarray*}
 in which ${\cal T}_{[t,T]}$ is the collection of $[t,T]$-valued stopping times. See \cite{bouchard2016fundamentals}.
 
 \begin{remark}\label{rem: V continuous} The fact that $(t,x)\in [0,T]\x \R^{d}_{+}\mapsto V(t,x)$ is continuous with polynomial growth follows from standard estimates under  the above assumptions. 
 In particular, the set $\{(t,x)\in [0,T]\x \R^{d}_{+}:V(t,x)=g(x)\}$ is closed. 
 \end{remark}

The aim of this section is to prove that $V$ is a viscosity solution of the non-linear parabolic equation 
\begin{equation}\label{eq: non lin pde}
\begin{array}{rl} 
r\vp-\mathcal{L}\vp-q(\cdot,\vp)=0  &\mbox{ on }  [0,T) \times (0,\infty)^{d}  \\
\vp(T,\cdot)=g & \mbox{ on }  (0,+\infty)^{d}, 
\end{array}
\end{equation}
for a suitable reaction function $q$ on $ (0,\infty)^{d}\x \R$. In the above,  $\Lc$ denotes the Dynkin operator associated to \eqref{eq: sde}: 
$$
\Lc \vp(t',x')=\partial_{t} \vp(t',x')+\langle rx' ,D\vp(t',x')\rangle+\frac12{\rm Tr}[\sigma\sigma^{\top}D^{2}\vp](t',x'),
$$
for a smooth function $\vp$. To be more precise, we define the function $q$ by 
\begin{align*}
q(x,y)=\left\{ \begin{array}{rcl}
0 & \mbox{if} & g(x)< y \\ 
c(x) & \mbox{if} & g(x)\geq y \\
\end{array}\right.,\;(x,y)\in  (0,\infty)^{d}\x \R, 
\end{align*}
where $c$ is a measurable map satisfying the following Assumption \ref{ass: Lg le c  defined}.

\begin{assumption}\label{ass: Lg le c  defined} The map $c: (0,\infty)^{d}\mapsto \R_{+}$ is continuous  with polynomial growth. Moreover,   $g$ is a viscosity subsolution  of 
$r \vp -\Lc \vp-c=0$ on ${\{(t,x)\in [0,T)\x (0,\infty)^{d}:V(t,x)=g(x)\}}$. 
\end{assumption}

Before providing examples of such a function $c$, let us make some important observations. 
 \begin{remark}\label{rem : c si payoff C2}
 First,  {$\{(t,x)\in [0,T)\x (0,\infty)^{d}:V(t,x)=g(x)\}\subset \{x\in(0,\infty)^{d}:$ $ g(x)>0\}$} if  $V>0$ on {$[0,T)\x (0,\infty)^{d}$}, which is typically the case in practice {(e.g.~because $g$ is  non-negative and the probability that $g(X)>0$ on $[0,T]$ is positive)}. In particular, if $g$ is $C^{2}$ on $\{g>0\}$ then one can choose $c=[rg-\Lc g]^{+}$ on $\{g>0\}$. Second, if $g$ is convex, then it can not be touched from above by a $C^{2}$ function at a point at which it is not $C^{1}$, {which implies that one can forget some singularity points in the verification of Assumption \ref{ass: Lg le c  defined} above}. {In Section \ref{sec: MC}, we shall suggest Monte-Carlo based numerical methods for the computation of $V$. One can then try to minimize the variance of the estimator over the choice of $c$. However, it seems natural to choose the function $c$ so that $g$ is actually a viscosity solution of $r \vp -\Lc \vp-c=0$ on ${\{(t,x)\in [0,T)\x (0,\infty)^{d}:V(t,x)=g(x)\}}$. In the numerical study of Section \ref{sec: MC}, this choice coincides with the $c$ with the minimal absolute value, which intuitively should correspond to the one minimizing the variance of the Monte-Carlo estimator. We leave the theoretical study of this variance minimization problem to future researches.}
 \end{remark}
 
\begin{example}\label{exam: function c} Let us consider the following examples in which $\bar \sigma$ is a constant matrix with $i$-th lines $\bar \sigma^{i}$. Fix $K,K_{1},K_{2}>0$ with $K_{1}<K_{2}$.
\begin{itemize}
\item For $d=1$ and a put $g:x\in (0,\infty)\mapsto [K-x]^{+}$,  the function $c$ is given by the constant $rK$. This is one of  the cases treated in \cite{benth2003semilinear}.  `
\item For $d=1$ and a strangle $g:x\in (0,\infty)\mapsto [K_{1}-x]^{+}+[x-K_{2}]^{+}$,   the function $c$ can be any continuous function equal to $rK_{1}$ on $(0,K_{1})$ and equal to $-rK_{2}$ on $(K_{2},\infty)$, {whenever $V>0$}.
\item   For $d=2$ and {a put on arithmetic mean} $g:x\in (0,\infty)^{2}\mapsto[K-\frac{1}{2} \displaystyle\sum_{i=1}^{2}x^{i}]^{+}$, we can take 
$
c=rK.
$
\item  For $d=2$ and {a put on geometric mean} $g:x\in (0,\infty)^{2}\mapsto[K-\sqrt{x^{1}x^{2}}]^{+}$, $c$ can be taken as 
\begin{eqnarray*}
x\in (0,\infty)^{2}\mapsto[r K-\frac{1}{8}(\|\bar \sigma^{1}\|^{2}+\|\bar \sigma^{2}\|^{2}-2\langle \bar \sigma^{1},\bar \sigma^{2}\rangle) \sqrt{x^{1}x^{2}}]^+.
\end{eqnarray*}
\end{itemize}
\end{example}

Since $q$ is discontinuous, we need to consider \eqref{eq: non lin pde} in the sense of viscosity solutions for discontinuous operators. More precisely, let $q_{*}$ and $q^{*}$ denote the lower- and upper-semicontinuous envelopes of $q$. We say that a lower-semicontinuous function $v$ is a viscosity supersolution of  \eqref{eq: non lin pde} if it is a  viscosity supersolution of
\begin{equation*}
\begin{array}{rl} 
r\vp-\mathcal{L}\vp-q_{*}(\cdot,\vp)=0  &\mbox{ on }  [0,T) \times (0,\infty)^{d}  \\
\vp(T,\cdot)=g & \mbox{ on }  (0,+\infty)^{d}.
\end{array}
\end{equation*}
Similarly, we say that a upper-semicontinuous function $v$ is a viscosity subsolution of  \eqref{eq: non lin pde} if it is a  viscosity subsolution of
\begin{equation*} 
\begin{array}{rl} 
r\vp-\mathcal{L}\vp-q^{*}(\cdot,\vp)=0  &\mbox{ on }  [0,T) \times (0,\infty)^{d}  \\
\vp(T,\cdot)=g & \mbox{ on }  (0,+\infty)^{d}.
\end{array}
\end{equation*}
We say that a continuous function is  a viscosity  solution of  \eqref{eq: non lin pde} if it is both a   viscosity super- and subsolution
\vspace{2mm}

Then, we have the following characterization of the American option price, which extends the result of  \cite{benth2003semilinear} to our context. Recall Remark \ref{rem: V continuous}

\begin{theorem}\label{thm: pde reaction diffusion}
Let $c$ be as in Assumption \ref{ass: Lg le c  defined}. Then, $V$ is a viscosity solution of  \eqref{eq: non lin pde}. It has a polynomial growth.
\end{theorem}

\begin{proof} {We just follow} the arguments of  \cite{benth2003semilinear}. 
\\
a. First note that $V\ge g$, so that\footnote{{Note that this is an important consequence of using $q_{*}$ instead of $q$.}} $q_{*}(\cdot,V)=0$. Hence, the supersolution property is equivalent to being a supersolution of 
\begin{equation*} 
r\vp-\mathcal{L}\vp=0  \;\mbox{ on }  [0,T) \times (0,\infty)^{d} \;\mbox{ and }\;  
\vp(T,\cdot)=g \; \mbox{ on }  (0,+\infty)^{d},
\end{equation*}
which is standard. 

b. Fix $(t,x)\in [0,T]\x (0,\infty)^{d}$ and a smooth function $\vp$ such that $(t,x)$ achieves a maximum on $ [0,T]\x (0,\infty)^{d}$ of $V-\vp$ and $(V-\vp)(t,x)=0$. If $t=T$, then the required result holds by {definition}. We now assume that $t<T$. If $(t,x)$ belongs to the open set $C:=\{V>g\}$, recall Remark \ref{rem: V continuous}, then one can find a $[t,T]$-valued stopping time $\tau$ such that $(\cdot \wedge \tau,X^{t,x}_{\cdot \wedge \tau})\in C$, and it follows from the dynamic programming principle, see e.g.~\cite{bouchard2011weak},  that 
$$  
\vp(t,x)\le  \E\left[e^{-r (\tau_{\eps}-t)}\vp(\tau_{\eps},X_{\tau_{\eps}})\right] 
$$
in which $\tau_{\eps}:=\tau \wedge (t+\eps)$ for $\eps>0$. Then,     standard arguments lead to  
$$
0\ge r\vp(t,x)-\Lc \vp(t,x)=r\vp(t,x)-\Lc \vp(t,x)-q^{*}(x,\vp(t,x)).
$$
 Let us now assume that $(t,x)\in S:=\{V=g\}$. {In particular, $\vp(t,x)=V(t,x)=g(x)$ and therefore $q^{*}(x,\vp(t,x))=q^{*}(x,V(t,x))=c(x)$}. Since $V\ge g$, $(t,x)$ is also a maximum of $g-\vp$ and  $\vp$ satisfies 
 $$
 0\ge r\vp(t,x)-\Lc \vp(t,x)-c(x)=r\vp(t,x)-\Lc \vp(t,x)-q^{*}(x,\vp(t,x)),
 $$
 by Assumption \ref{ass: Lg le c  defined}.
\end{proof}

This viscosity solution property can be complemented with a comparison principle as in \cite{benth2003semilinear}.  Combined with Theorem \ref{thm: pde reaction diffusion}, it shows that $V$ is the unique viscosity solution of \eqref{eq: non lin pde} with polynomial growth.

\begin{proposition}\label{prop: comp} Let the conditions of Theorem \ref{thm: pde reaction diffusion} hold. Let $v$ and $w$ be respectively a super- and a subsolution of  \eqref{eq: non lin pde}, with polynomial growth. Then, $v\ge w$ on $[0,T]\x (0,\infty)^{d}$.
\end{proposition}

\begin{proof}  We adapt the arguments of \cite{benth2003semilinear}. As usual, one can assume without loss of generality that $r>0$, upon replacing $v$ by $(t,x)\mapsto e^{-\rho t}v(t,x)$ and $w$ by $(t,x)\mapsto e^{-\rho t}w(t,x)$ for some $\rho>|r|$. Fix $p\ge 1$ and $C>0$ such that $|v(t,x)|+|w(t,x)|\le C(1+\|x\|^{p})$ for all $(t,x)\in [0,T]\x(0,\infty)^{d}$. Set $\psi(t,x):=e^{-\kappa t}(1+\|x\|^{2p})$ for $(t,x)\in  [0,T]\x (0,\infty)^{d}$, for some $\kappa$ large enough so that $\psi$ is a supersolution of $-\Lc\vp=0$ on $[0,T)\x (0,\infty)^{d}$, which is possible since $\bar \sigma$ is bounded. Set 
$$
\phi^{\eps}_{n}(t,x,y):=w(t,y)-v(t,x)-n\|x-y\|^{2p}-{\lambda}\psi(t,y)-\frac{\eps}{\prod_{i=1}^{d} x^{i}}-\frac{\eps}{\prod_{i=1}^{d} y^{i}}
$$
for $n\ge 1$, $\eps>0$, $(t,x,y)\in  [0,T]\x (0,\infty)^{2d}$, and a given ${\lambda}>0$. Assume that  $\sup_{[0,T]\x(0,\infty)^{2d}}$ $(w$ $-v)>0$. Then one can find $\eps_{\circ},{\lambda}>0$  and $\delta>0$ such that 
\begin{align}\label{eq: Phineps ge delta}
\sup_{[0,T]\x(0,\infty)^{2d}}\phi^{\eps}_{n}\ge \delta,\;\mbox{ for $\eps\in (0,\eps_{\circ})$ and $n\ge 1$.}
\end{align}
Clearly, $\phi^{\eps}_{n}$ admits a maximum point  $(t_{n}^{\eps},x_{n}^{\eps},y_{n}^{\eps})$ on $[0,T{]} \x (0,\infty)^{2d}$. Moreover, it follows from standard arguments that $(t_{n}^{\eps},x_{n}^{\eps},y_{n}^{\eps})$ converges to some $(t_{n},x_{n},y_{n})\in [0,T{]}\x \R_{+}^{d}$ as $\eps\to 0$, possibly along a subsequence, and that 
\begin{align}\label{eq: proof comp conv n eps}
&\lim_{\eps\to 0}( \frac{\eps}{\prod_{i=1}^{d} (x_{n}^{\eps})^{i}}+\frac{\eps}{\prod_{i=1}^{d} (y_{n}^{\eps})^{i}})=0\;\mbox{ , }\; \lim_{n\to \infty}n\|x_{n}-y_{n}\|^{2p}=0,
\\
&\lim_{\eps\to 0}(w(t_{n}^{\eps},y_{n}^{\eps}),v(t_{n}^{\eps},x_{n}^{\eps}))=(w(t_{n}^{},y_{n}^{}),v(t_{n}^{},x_{n}^{}))\label{eq: lim eps in v w}
\\
&\lim_{n\to \infty} y_{n}=\hat y, \mbox{ for some } \hat y \in \R^{d}_{+},\label{eq:xn conv}
\end{align}
possibly along subsequences, see e.g.~\cite[Proof of Theorem 4.5]{bouchard2016fundamentals} and \cite{crandall1992user}. 
Combining  Ishii's Lemma, see e.g.~\cite{crandall1992user},   with the super- and subsolution properties of $v$, $\psi$ and $w$, we obtain 
\begin{align*}
0\ge& r(w(t_{n}^{\eps},y_{n}^{\eps})-v(t_{n}^{\eps},x_{n}^{\eps}))-q^{*}(y_{n}^{\eps},w(t_{n}^{\eps},y_{n}^{\eps}))+q_{*} (x_{n}^{\eps},v(t_{n}^{\eps},x_{n}^{\eps}))\\
&-O (n\|x_{n}^{\eps}-y_{n}^{\eps}\|^{2p})-\eta^{n}_{\eps}
\end{align*}
in which, thanks to the left-hand side of \eqref{eq: proof comp conv n eps}, $\eta^{n}_{\eps}\to 0$ as $\eps\to 0$, for all $n\ge 1$. By the right-hand side of \eqref{eq: proof comp conv n eps}, the discussion just above it, and \eqref{eq: lim eps in v w}, sending $\eps\to 0$ and then $n\to \infty$  leads to 
\begin{align*}
0\ge& \limsup_{n\to \infty }\left\{r(w(t_{n},y_{n})-v(t_{n},x_{n}))-q^{*}(y_{n},w(t_{n},y_{n}))+q_{*} (x_{n},v(t_{n},x_{n}))\right\}
\end{align*}
and therefore  
\begin{align*}
\liminf_{n\to \infty } \{q^{*}(y_{n},w(t_{n},y_{n}))-q_{*} (x_{n},v(t_{n},x_{n}))\}&\ge r\delta
\end{align*}
by  \eqref{eq: Phineps ge delta}. 
Recall that $c$ is non-negative and that $w(t_{n},y_{n})-v(t_{n},x_{n})\ge \delta$ by \eqref{eq: Phineps ge delta}. If, along a subsequence, $g(x_{n})>v(t_{n},x_{n})$ for all $n$, then $q^{*}(y_{n},w(t_{n},y_{n}))-q_{*} (x_{n},v(t_{n},x_{n}))\le c(y_{n})-c(x_{n})$  for all $n$, leading to a contradiction since $c(x_{n})-c(y_{n})\to 0$ as $n\to \infty$ (recall \eqref{eq: proof comp conv n eps} and \eqref{eq:xn conv}) and $r>0$. If, along a subsequence, $g(x_{n})\le v(t_{n},x_{n})$ for all $n$, then $g(y_{n})\le  v(t_{n},x_{n})+\delta/2\le w(t_{n},y_{n})-\delta/2$ for all $n$ large enough and the above liminf is also non-positive. A contradiction too.
\end{proof}

\section{Monte-Carlo estimation}\label{sec: MC}
\def\E{{\mathbb E}}
\def\Lb{{\mathbf L}}
\def\Sb{{\mathbf S}}
\def\Fc{{\cal F}}
The solution of \eqref{eq: non lin pde} is formally related to the solution  $(Y,Z)\in \Sb_{2}\x \Lb_{2}$ of the backward stochastic differential equation
 $$
Y=e^{-rT}g(X_{T})+\int_{\cdot}^{T} e^{-rs}q(X_{s},e^{rs}Y_{s})ds-\int_{\cdot}^{T} Z_{s}dW_{s}
 $$
 by $e^{-r\cdot}V(\cdot,X)=Y$. In the above, $\Sb_{2}$ denotes the space of adapted processes $\xi$ such that $\E[\sup_{[0,T]}\|\xi\|^{2}]<\infty$ and   $\Lb_{2}$ denotes the space of predictable processes $\xi$ such that $\E[\int_{0}^{T}\|\xi_{t}\|^{2}dt]<\infty$. 
 
 \begin{remark} Note that, if $(Y,Z)$ satisfies the above BSDE, then
$$
 Y_{0}=\E[e^{-rT}g(X_{T})+\int_{0}^{T} e^{-rs}q(X_{s},e^{rs}Y_{s})ds].
 $$
 In the case where $c=rg-\Lc g$, on ${\{(t,x)\in [0,T)\x (0,\infty)^{d}: V(t,x)=g(x)\}}$, this corresponds to the early exercise premium formula. Recall Assumption \ref{ass: Lg le c  defined} and see \cite[Section 6]{benth2003semilinear}.
 \end{remark}
 
  In practice the above BSDE is not well-posed because $q$ is not continuous. However, it can be smoothed out for the purpose of numerical approximations. In the following, we write $\E_{s}[\cdot]$ to denote the expectation given $\Fc_{s}$, $s\le T$.
 
 \begin{proposition}\label{prop : stability qn} Let the condition of Theorem \ref{thm: pde reaction diffusion} hold. Let $(q_{n})_{n\ge 1}$ be a sequence of continuous functions on $ (0,\infty)^{d}\x \R$ that are Lipschitz in their last component\footnote{{See below for examples.}}. Assume that $(q_{n})_{n\ge 1}$ is uniformly bounded by a function with polynomial growth in its first component and linear growth in its last component. Assume further that  
 \begin{align}\label{eq: relaxed lim qn}
 \limsup_{\tiny\begin{array}{c}n\to \infty\\ (x',y')\to (x,y)\end{array}}q_{n}(x',y'){\le}q^{*}(x,y)\; \mbox{ and } \liminf_{\tiny\begin{array}{c}n\to \infty\\ (x',y')\to (x,y)\end{array}}q_{n}(x',y'){\ge}q_{*}(x,y)
 \end{align}
 for all $(x,y)\in (0,\infty)^{d}\x \R$. For $(t,x)\in [0,T]\x (0,\infty)^{d}$, let  $(Y^{t,x,n})_{n\ge 1}$ be such that 
 $$
 Y^{t,x,n}_{s}=\E_{s}[e^{-r{T}}g(X^{t,x}_{T})+\int_{s}^{T}e^{-r{u}} q_{n}(X^{t,x}_{u},e^{{r{u}}}Y^{t,x,n}_{u})du] ,
 $$
 for $s\in [t,T]$, 
 and set $V_{n}(t,x):=e^{rt}Y^{t,x,n}_{t}$. Then, $(V_{n})_{n\ge 1}$ converges pointwise to $V$ as $n\to \infty$. 
 \end{proposition} 
 
\begin{proof} Each BSDE associated to $q_{n}$ admits a unique solution $(Y^{t,x,n},Z^{t,x,n})\in \Sb_{2}\x \Lb_{2}$, and it is standard to show that $V_{n}$ is a continuous viscosity solution of 
 $$
 r\vp -\Lc\vp-q_{n}(\cdot,\vp)=0 \mbox{ on } [0,T)\x (0,\infty)^{d} \mbox{ and } \vp(T,\cdot)=g  \mbox{ on } (0,\infty)^{d}. 
 $$
 Moreover, $(V_{n})_{n\ge 1}$ has (uniformly) polynomial growth, thanks to the uniform polynomial growth assumption on $(q_{n})_{n\ge 1}$. See e.g.~\cite{pardoux1998backward}. 
 By stability and \eqref{eq: relaxed lim qn}, see e.g.~\cite{barles1994solutions}, it follows that the relaxed limsup $V^{*}$ and liminf $V_{*}$ of $(V_{n})_{n\ge 1}$ are respectively sub- and super-solutions of 
 \eqref{eq: non lin pde}. By Proposition \ref{prop: comp}, $V^{*}\le V\le V_{*}$ and therefore equality holds among the three functions. 
\end{proof}

Therefore, up to a smoothing procedure, we are back to essentially solving a BSDE. In the next two sections, we propose two approaches. The first one consists in smoothing $q$ into a a smooth function $q_{n}$ to which we apply the local polynomial approximation procedure of \cite{bouchard2016numerical}. This allows us to use a pure forward Monte-Carlo method for the estimation of $V_{n}$, based on branching processes. In the second approach, we only add an independent noise in the definition of $q$, which also has the effect of smoothing it out, and then use a very simple version of  the algorithm in \cite{bouchard2016numerical}. As our numerical experiments show, the first approach is quite unstable while the second one is very efficient.

\subsection{Local polynomial approximation and branching processes}\label{sec: loc poly approach}
 
Given Proposition \ref{prop : stability qn}, it is tempting to estimate the American option price by using the recently developed Monte-Carlo method for BSDEs, see \cite{bouchard2012monte} and the references therein.  Here, we propose to use the forward approach suggested by \cite{bouchard2016numerical}, which is based on the use of branching processes coupled (in theory) with Picard iterations. 

The first step consists in approximating the Heaviside function $H:z\mapsto \1_{\{z\ge 0\}}$ by a sequence of  Lipschitz functions $(H_{n})_{n\ge 1}$ and to define $q_{n}$ by 
$$
q_{n}:(x,y)\mapsto c(x)H_{n}(g(x)-y). 
$$  
Then, $q_{n}$ is approximated by a map ${(x,y)}\mapsto \bar q_{n}(x,y,y)$ of local polynomial form: 
\begin{equation}\label{eq: local poly}
\bar q_{n}:(x,y,y') \rightarrow \displaystyle\sum_{j=1}^{j_{0}}\sum_{l=0}^{l_{0}}a_{j,l}(x)y^{l}\phi_{j}(y')  
\end{equation}
where $(a_{j,l}, \phi_{j})_{l\leq l_{0},j\leq j_{0}}$ is a family of continuous and bounded maps satisfying 
\begin{equation*}
|a_{j,l}|\leq C_{l_{0}},\,\,\,|\phi_{j}(y'_{1})-\phi_{j}(y'_{2})|\leq L_{\phi}|y'_{1}-y'_{2}|\; \mbox{ and } |\phi_{j}|\leq 1,
\end{equation*}
for all $y'_{1}$,$y'_{2} \in \mathbb{R}$, $j\leq j_{0}$ and  $l \leq  l_{0}$, for some constants $C_{l_{0}}$, $L_{\phi}\geq 0$.
The elements of  $(a_{j,l}(x))_{l\le l_{0}}$ should be interpreted as the coefficients of a polynomial approximation of ${q_{n}}$ on a subset $A_{j}$, in which $(A_{j})_{j\le j_{\circ}}$   forms a partition of $\mathbb{R}$ and the $\phi_{j}$'s as smoothing kernels that allow one to pass in a Lipschitz way from one part of the partition to another one, see \cite{bouchard2016numerical}. 

  Then, one can consider the sequence of BSDEs
\begin{align*}
\bar  Y^{t,x,n,k+1}_{s}=&\E_{s}[e^{-rT}g(X^{t,x}_{T})]\\
&+\E[\int_{s}^{T}e^{-ru} \bar q_{n}(X^{t,x}_{u},e^{ru}\bar Y^{t,x,n,k+1}_{u},e^{ru}\bar Y^{t,x,n,k}_{u})du],\;k\ge 1,
\end{align*}
with $\bar  Y^{t,x,n,1}$ given as an  initial prior (e.g.~$e^{r\cdot}g(X^{t,x})$). Given $\bar Y^{t,x,n,k}$, $\bar Y^{t,x,n,k+1}$ solves a BSDE with polynomial driver that can be estimated by using branching processes as in the Feynman-Kac representation of the Kolmogorov-Petrovskii-Piskunov equation, see  \cite{henry2012cutting,henry2016branching}. We refer to \cite{bouchard2016numerical} for more details. 
\vspace{2mm}

In practice, we use the Method A of \cite[Section 3]{bouchard2016numerical}. We perform a numerical experiment in dimension $1$, with a    time horizon of one year, and a risk-free interest rate set at $6\%$. We consider the Black and Scholes model with one  single stock whose volatility is $40\%$. We price a put option which strike is {$K:=40$}. At the money, the American option price is around $5.30$, while the European option is worth $5.05$. In view of  Example \ref{exam: function c}, we take $c=rK$\footnote{{Note that, for this payoff, the constant $rK$ is the function with the smallest absolute value among the functions $c$ satisfying the requirements of Assumption \ref{ass: Lg le c  defined}}.}. We first smooth the driver with a centered Gaussian density with variance $\kappa^{-2}$, so as to replace it by $0.5r K e^{-r t}\mbox{erfc}(\kappa*{(y-e^{-rt}g(x))})$ with $\kappa=10$. See Figure \ref{figure_driver}. Then, we apply a  quadratic  spline approximation.  
{ In actual computation, as it is impossible to apply spline approximation on the whole half real line, we limited the domain of y for the driver function to $[0, 40(1 - e^{-0.06})]$.
We partition this bounded domain into 20 intervals with equal-distant points and define a piecewise polynomial on this domain by assigning a quadric polynomial to each intervals.
Finally, we match the values and derivatives of our piecewise polynomial at each point of the grid to the original function (except at the right-end where the derivative is assumed to be zero).
The truncation of domain will not alter the computational result as our limited domain includes the maximum payoff for the put option.
The resulting approximation is indistinguishable from the original function displayed on Figure \ref{figure_driver}.}
\\
We also partition $[0,T]$ in $10$ periods. As for the grid in the $x$-component, we use a $25$-point uniform space-grid on the interval $[e^{-20}, 80]$.
\\
We estimate the early exercise value {by first using   1.000 Monte-Carlo paths}. As can be seen on Figure \ref{figure_montecarlo2}, the results are not good {and this does not improve much with a higher number of simulations}. The algorithm {turns out to be} quite unstable and not accurate.  It remains pretty unstable even for a large number of simulated paths. This is not so surprising.
 Indeed, as explained in \cite{bouchard2016numerical}, their approach is dedicated to situations where the driver functions is rather smooth, so that the local polynomial's coefficients $(a_{j,l})_{j,l}$ are small, and the supports of the $\phi_{j}$'s are large and do not intersect too much.
Since we are approximating the Heaviside function, none of these requirements are met.
\begin{figure}
\begin{center}
\includegraphics[height=0.2\textwidth, width = 0.5\textwidth]{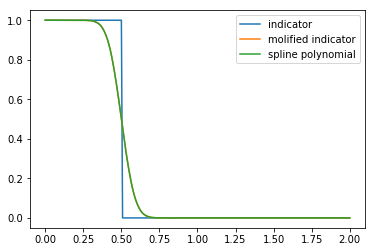}
\caption{Approximation of  {$y\mapsto  \1_{\{y\le 0.5\}}$}.}
\end{center}
\label{figure_driver}
\end{figure}
\begin{figure}
\begin{center}
\includegraphics[width = 0.5\textwidth]{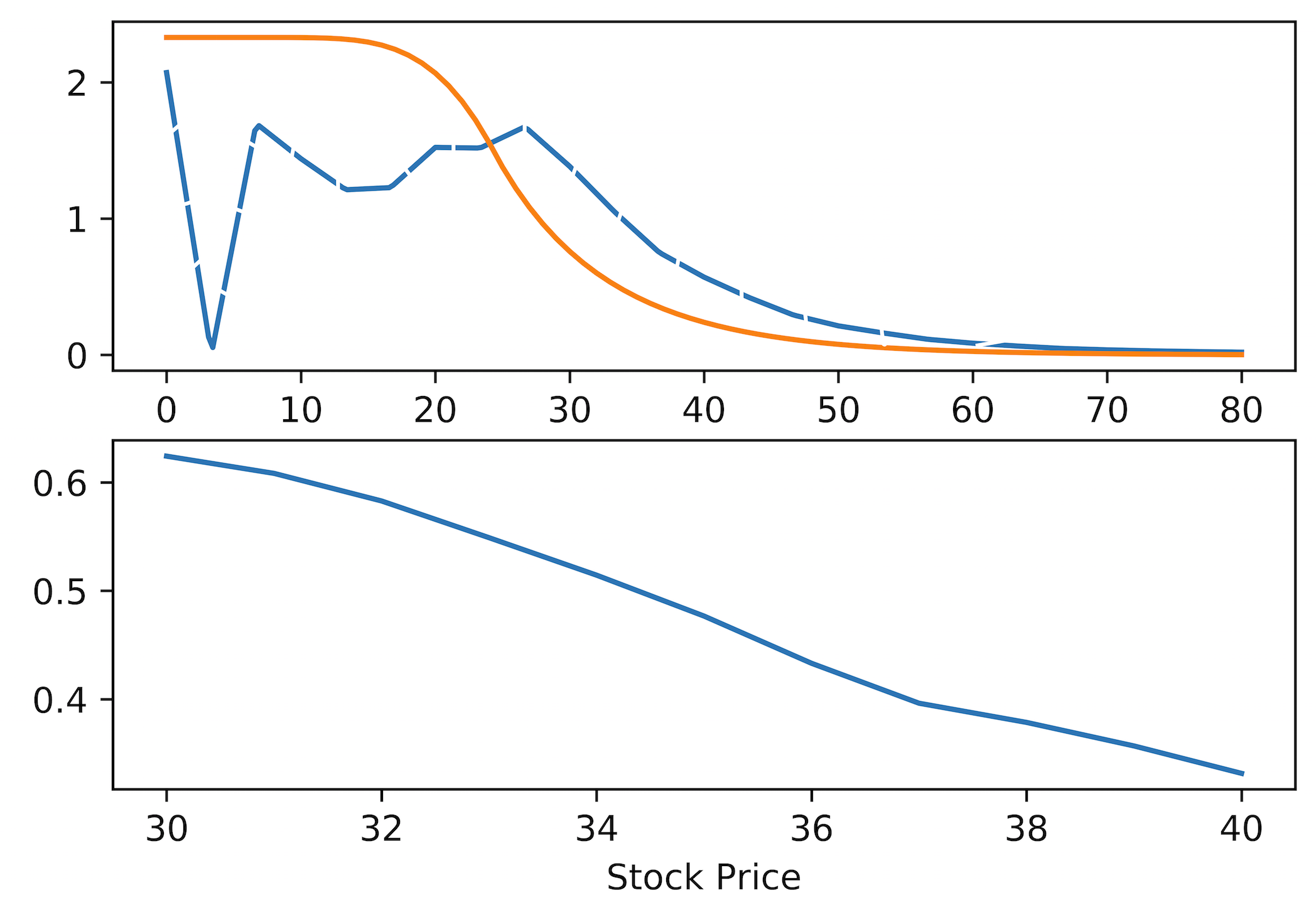}
\end{center}
\caption{Branching with local polynomial approximation. Upper graph: Early exercise premium (plain line obtained by a pde solver, dashed line estimated). Lower graph: Error on the early exercise premium estimation.}
\label{figure_montecarlo2}
\end{figure}

\newpage

 \subsection{Driver randomization}\label{sec: indic approach}
 
 In this second approach, we enlarge the state space so as to introduce an independent   integrable random variable  $\epsilon$ with density $f$ such that $z\mapsto (1+|z|)f'(z)$ is integrable. 
 We assume that the interior of the support of $f$ is of the form $(m_{\epsilon},M_{\epsilon})$ with $-\infty\le m_{\epsilon}<M_{\epsilon}\le \infty$. 
 Then, we define the sequence of random maps
 $$
 \tilde q_{n}(x,y):=c(x)\1_{\{g(x)+\frac{\epsilon}{n}\ge   y \}}
 $$ 
 as well as 
 \begin{align*}
 q_{n}(x,y):=&c(x)n\left\{ [g(x)+M_{\epsilon}/n-y]^{+} f(M_{\epsilon})-[g(x)+m_{\epsilon}/n-y]^{+} f(m_{\epsilon})\right\} \\
 &-c(x)n\int [g(x)+z/n-y]^{+} f'(z)dz
 \end{align*}
  so that 
 $$
 q_{n}(x,y)=\E[ \tilde q_{n}(x,y)]
 $$
 for $n\ge 1$. If $c$ is non-negative, continuous and  has polynomial growth, then the sequence $(q_{n})_{n\ge 1}$ matches the requirements of  Proposition \ref{prop : stability qn}. 
 
 We now let $\tau$ be an independent exponentially distributed random variable with density $\rho$ and cumulative distribution $1-\bar F$. Then, $Y^{t,x,n}$ defined as in  Proposition \ref{prop : stability qn} satisfies 
\begin{align*}
Y^{t,x,n}_{s}=&\E_{s}\left[e^{-rT}\frac{g(X^{t,x}_{T})}{\bar F(T-t)}{\bf 1}_{\{T-t\le \tau\}}+{\bf 1}_{\{T-t>\tau\}}\frac{e^{-r\tau}\tilde q_{n}(X^{t,x}_{t+\tau},e^{r\tau}Y^{t,x,n}_{t+\tau})}{{\rho(\tau)}}\right].
\end{align*}
This can be viewed as a branching based representation in which particles  die at an exponential time. When a particle die before $T$, we give it the (random)  mark  $\tilde q_{n}(X^{t,x}_{t+\tau},e^{r\tau}Y^{t,x,n}_{t+\tau})$. In terms of the representation of Section \ref{sec: loc poly approach}, this corresponds to $j_{0}=1,$ $l_{0}=0$, to replacing $a_{1,0}(x)\phi_{1}(y')$ by $\tilde q_{n}(x,y')$, and to not using a Picard iteration scheme.

 On a finite time grid $\pi\subset [0,T]$ containing $\{0,T\}$, it can be approximated by the sequence $v^{\pi}_{n}$ defined by $v^{\pi}_{n}(T,\cdot)=g$ and 
\begin{align}
v^{\pi}_{n}(t,x)=&\E\left[e^{-rT}\frac{g(X^{t,x}_{T})}{\bar F(T-t)}{\bf 1}_{\{T-t\le \tau\}}\right]\label{eq: algo random}\\
&+\E\left[{\bf 1}_{\{T-t>\tau\}}\frac{e^{-r\tau}\tilde q_{n}(X^{t,x}_{\phi^{\pi}_{t+\tau}},e^{r\tau}v^{\pi}_{n}(\phi^{\pi}_{t+\tau},X^{t,x}_{\phi^{\pi}_{t+\tau}}))}{{\rho(\tau)}}\right], \nonumber
 \end{align}
 where $\phi^{\pi}_{s}:=\inf\{s'\ge s: s'\in \pi\}$ for $s\le T$. Showing that $v^{\pi}_{n}(\phi^{\pi}_{t},x)$ converges point-wise to $Y^{t,x,n}_{t}$ as the modulus of $\pi$ vanishes can be done by working along the lines of \cite[Section 4.3]{baradel2016optimal} or \cite{fleming1989existence}. In view of Proposition \ref{prop : stability qn}, $v^{\pi}_{n}$ converges point-wise to $V$ as $|\pi|\to 0$ and $n\to \infty$. A similar analysis could be performed when considering a grid in space, which will be necessary in practice. 
 \vspace{2mm}
 
Then, \eqref{eq: algo random} provides a natural backward algorithm: given a space-time grid $\Pi:=(t_{i},x_{j})_{i,j}$,  \eqref{eq: algo random} can be used to compute $v^{\pi}_{n}(t_{i},x_{j})$ given the already computed values of $v^{\pi}_{n}$ at the later times in the grid, by replacing the expectation by a Monte-Carlo counterpart.

\begin{figure}[h]
\begin{center}
\includegraphics[height = 0.9\textheight, width=0.9\textwidth]{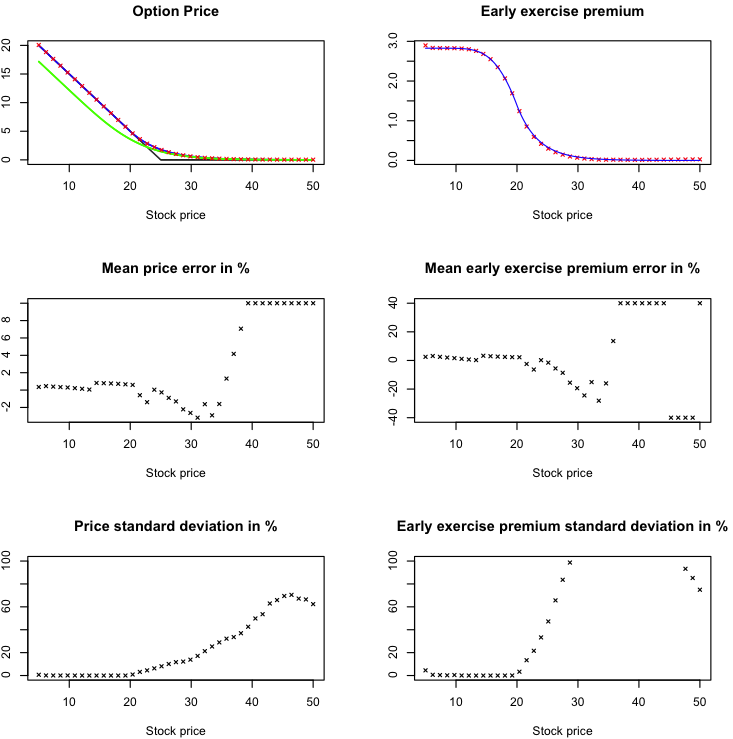}
\end{center}
\caption{Branching with indicator driver. Put option, 1.000 sample paths. Plain lines=true values, crosses=estimations.}
\label{figureput1000}
\end{figure}

\begin{figure}[h]
\begin{center}
\includegraphics[height =0.9 \textheight, width =0.9 \textwidth]{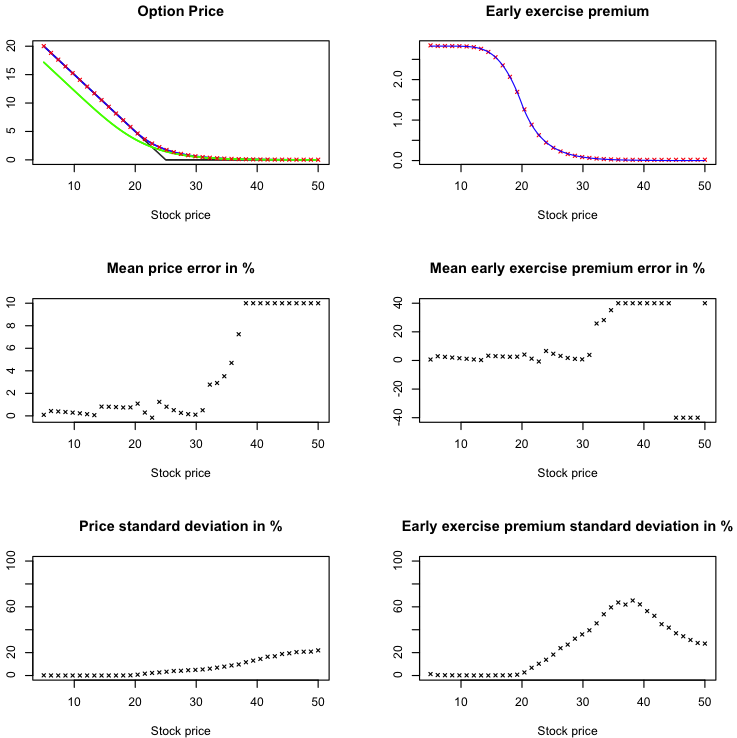}
\end{center}
\caption{Branching with indicator driver. Put option, 10.000 sample paths. Plain lines=true values, crosses=estimations.}
\label{figureput10000}
\end{figure}

\begin{figure}[h]
\begin{center}
\includegraphics[height = 0.9\textheight, width = 0.9\textwidth]{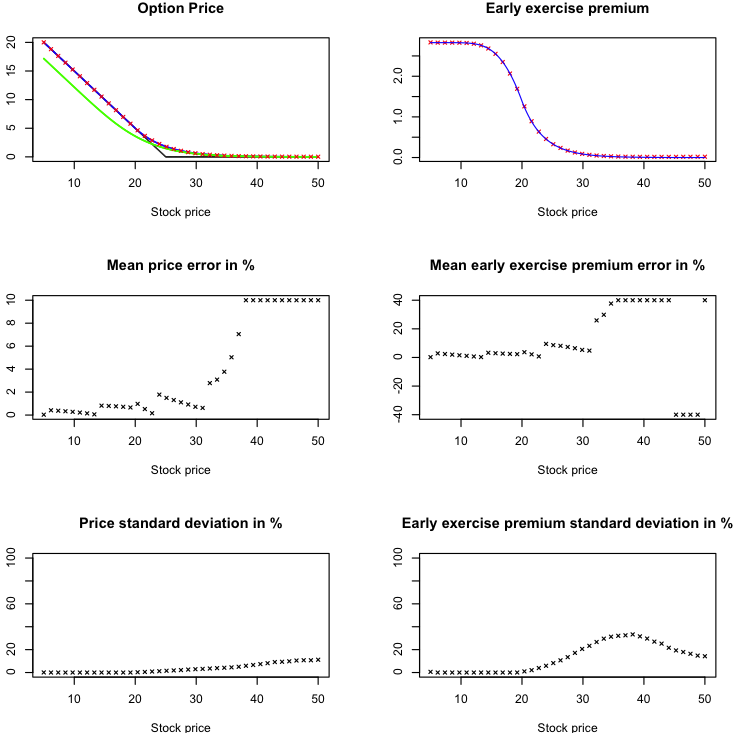}
\end{center}
\caption{Branching with indicator driver. Put option, 50.000 sample paths. Plain lines=true values, crosses=estimations.}
\label{figureput50000}
\end{figure}

Let us {now} consider {a {put} option pricing problem within the Black-Scholes model as in  the previous section. The interest rate is $6\%$, the volatility is $20\%$ and the strike is $25$}. {The partition $\pi$ of $[0,T]$ is uniform with  $100$ time steps. However, we update $v^{\pi}_{n}$ only every 10 time steps (and consider that it is constant in time in between). The fine grid $\pi$ is therefore only used to approximate $X^{t,x}_\tau$ by $X^{t,x}_{\phi^{\pi}_{\tau}}$ accurately.}
{We  use a $40$-points equidistant space-grid on the interval $[5, 50]$. The random variable ${\epsilon/n}$ is exponentially distributed, with mean equal to $10^{-100}$, while $\tau$ has mean $0.6$.  In Figures \ref{figureput1000}, \ref{figureput10000} and \ref{figureput50000}, we provide the estimated prices, the estimated early exercise premium  as well as the corresponding relative errors. The statistics are based on 50 independent trials. The reference values are computed with an implicit scheme for the associated pde, with regular grids of 500 points in space and 1.000 points in time (we also provide the European option price in the top-left graph, for comparison). The relative errors are capped to $10\%$ or $40\%$ for ease of readability. These graphs show that the numerical method is very efficient. The relative error for a stock price higher that 30/35 are not significant since it corresponds to option prices very close to $0$. For $10.000$ simulated paths, it takes 12 secondes for one estimation of the whole price curve with a R code running on a Macbook 2014, 2.5 GHz Intel Core i7, with 4 physical cores.}

\begin{figure}[h]
\begin{center}
\includegraphics[height = 0.9\textheight, width = 0.9\textwidth]{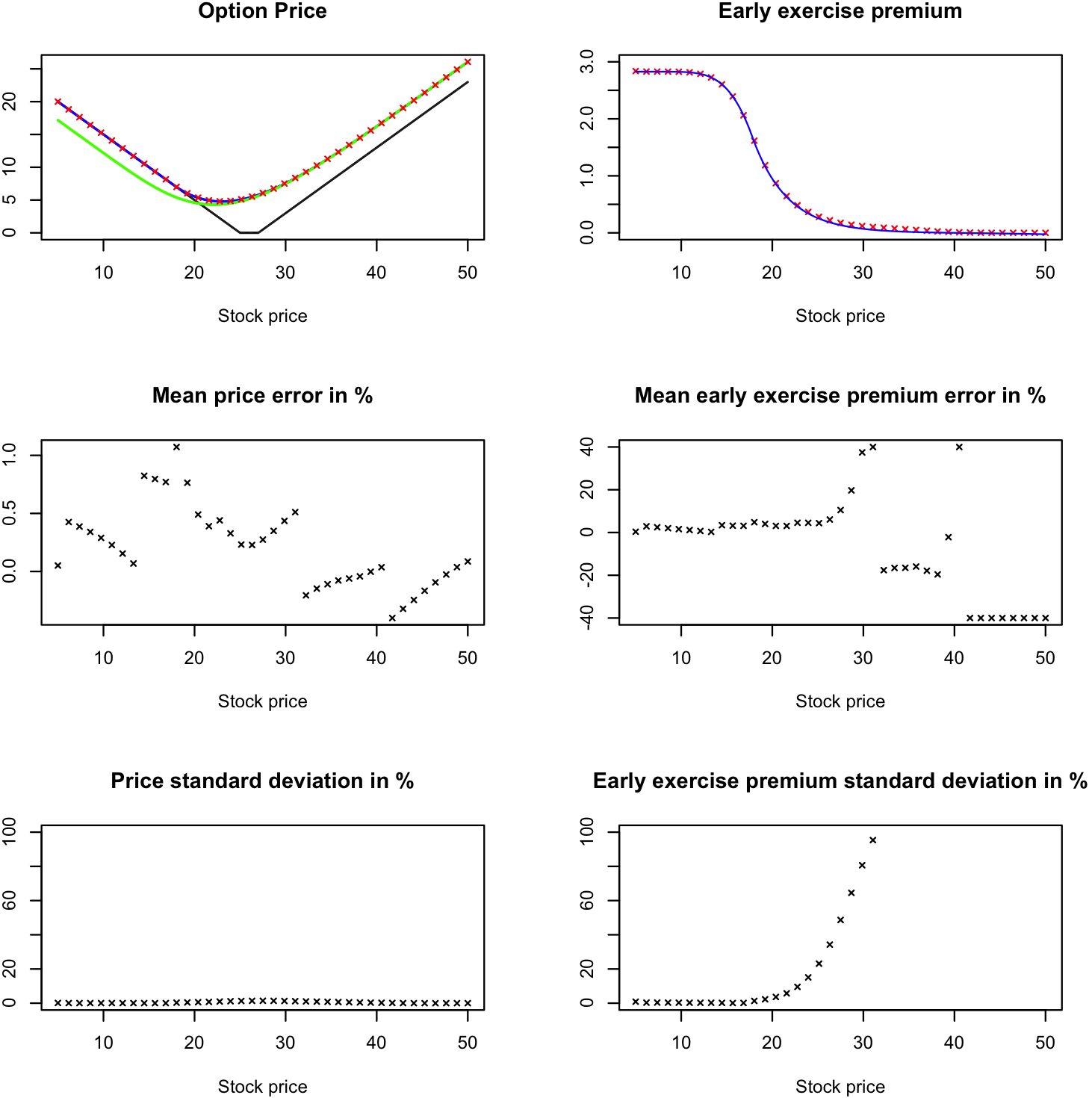}
\end{center}
\caption{Branching with indicator driver. Strangle option, 50.000 sample paths. Plain lines=true values, crosses=estimations.}
\label{figurestrangle50000}
\end{figure}

{We next consider a strangle with strikes {25} and 27, see Example \ref{exam: function c}.  The results obtained with 50.000 sample paths are displayed in Figures \ref{figurestrangle50000}.}

{Note that we do not use any variance reduction technique in these experiments. }
  \bibliographystyle{plain}

\end{document}